\documentclass{amsart}
\usepackage{amscd}
\usepackage{mathrsfs}
\usepackage{cases}
\usepackage{latexsym}
\usepackage{amsmath}
\usepackage{indentfirst}
\usepackage[arrow,matrix]{xy}
\usepackage{stmaryrd}
\usepackage{amsfonts, color}
\usepackage{amsmath,amssymb,amscd,bbm,amsthm,mathrsfs,dsfont}
\usepackage{graphicx}
\usepackage{fancyhdr}
\usepackage{amsxtra,ifthen}
\usepackage{verbatim}
\usepackage{latexsym,epsfig,amsmath}
\usepackage{epsfig}

\DeclareMathOperator{\End}{End} \DeclareMathOperator{\ch}{char}

\DeclareMathOperator{\rad}{rad}

\numberwithin{equation}{section}

\theoremstyle{plain}
\newtheorem{theorem}{Theorem}[section]
\newtheorem{lemma}[theorem]{Lemma}
\newtheorem{proposition}[theorem]{Proposition}

\newtheorem{conjecture}[theorem]{Conjecture}

\theoremstyle{definition}

\newtheorem{remark}[theorem]{Remark}

\begin{document}

\title[Finite dimensional zJpd algebras are generated by idempotents]
{Finite dimensional zero Jordan product determined algebras are generated by idempotents}

\author{Hongyu Jia  and Zhankui Xiao}

\address{Jia: School of Mathematical Sciences, Huaqiao University,
Quanzhou, Fujian, 362021, P. R. China}
\email{jiahy1995@163.com}

\address{Xiao: School of Mathematical Sciences, Huaqiao University,
Quanzhou, Fujian, 362021, P. R. China}

\email{zhkxiao@hqu.edu.cn}

\begin{abstract}
Bre\v{s}ar showed that a finite dimensional unital associative algebra is zero product determined if and only if it is generated by idempotents. For the analogue of zero Jordan product determined algebras, only one direction was known: over a field of characteristic not 2, every algebra generated by idempotents is zero Jordan product determined. Whether the converse holds has remained an open problem.

In this paper, we answer this question affirmatively in the finite dimensional case.
Some related open problems are stated at the end.
\end{abstract}

%16P10  	Finite rings and finite-dimensional associative algebras
%16N40  	Nil and nilpotent radicals, sets, ideals, associative rings
%16K20  	Finite-dimensional division rings

\thanks{The first author is supported by the National Natural Science Foundation of China (No.12401022) and the Natural Science Foundation of Xiamen (No.3502Z202471031).}
\maketitle

\section{Introduction}\label{xxsec1}

Let $F$ be an arbitrary field. An $F$-algebra is a pair $(A,\cdot)$, where $A$ is an $F$-linear space 
and $\cdot:A\times A\rightarrow A$ is an $F$-bilinear map. 
We sometimes write an $F$-algebra $A$ for short if there is no confusion on the binary operator ‘‘$\cdot$’’. 
An $F$-algebra  $A$ is called {\em zero product determined}
({\em zpd} for short) if for every bilinear map $f: A\times A\rightarrow F$ with the property that
\begin{equation}\label{eq1}
f(a,b)=0\quad {\rm whenever}\quad a\cdot b=0,
\end{equation}
then there exists a linear map $\Phi:A\rightarrow F$ such that
\begin{equation}\label{eq2}
f(x,y)=\Phi(x\cdot y)\quad {\rm for\ all}\quad x,y\in A.
\end{equation}

Let $A$ be an associative algebra over $F$. If we introduce the Jordan
product on $A$ as usual by $x\circ y=xy+yx$, then $(A,\circ)$
becomes a Jordan algebra. We say that $A$ is \emph{zero Jordan product determined}
(\emph{zJpd} for short) if the Jordan algebra $(A,\circ)$ is zero product determined.

The concept of a zpd algebra originated from the field of linear preserver problems
(see \cite{M} for instance).
In \cite{BS06}, Bre\v{s}ar and \v{S}emrl investigated commutativity preserving linear maps on central simple algebras, and in \cite{ABEV07},  Alaminos, Bre\v{s}ar, Extremera and Villena studied zero product preserving maps and local derivations on $C^{*}$-algebras. 
Both papers share a core strategy: establishing results about certain linear maps by investigating bilinear maps that vanish on pairs of elements with zero product.
Motivated by the aforementioned works, the concept of a zpd algebra was introduced in \cite{BGS}. 
Since then, zpd algebras have been extensively studied, see \cite{ABESV,ABEV,BGLLZ,BH,CKLW,Grasic10,Grasic11,Grasic15,WangYC,YLL}
and the references therein. 
We refer the reader to Bre\v{s}ar's monograph \cite{BOOK} for a comprehensive understanding of the subject.

In \cite{Bresar}, Bre\v{s}ar studied the universal properties of the structure
of associative zpd algebras, and established that a finite dimensional unital associative algebra over a field is zpd if and only if it is generated by idempotents. 
On the other hand, deciding whether an algebra is generated by idempotents is itself a highly nontrivial matter, which has received a lot of attention in the field of representation theory.
By introducing zero action determined modules, Hu and the second author \cite{HX} provided a module theoretic criterion for testing an algebra generated by idempotents. 

Motivated by the progress of associative zpd algebras, 
researchers began to investigate the analogous description of zJpd algebras. 
A step in this direction was made in \cite{MR3325219} 
(and partially in \cite{Ghahramani}), where the following result was established.

\begin{theorem}\label{thm13}
Let $A$ be an associative unital algebra over a field of characteristic not 2. 
If $A$ is generated by idempotents, then $A$ is zJpd.
\end{theorem}

However, the converse implication of Theorem \ref{thm13} remains open up to now, as far as we know.
Discussing this problem, in \cite[Page 54]{BOOK} Bre\v{s}ar remarked:
\begin{quote}
``As a matter of fact, we are facing the same problem as with zpd algebras: 
we do not know of any unital zJpd algebra 
(over a field of characteristic not $2$) that is not generated by idempotents. 
Generally speaking, our understanding of zJpd algebras is, at present, rather limited.''
\end{quote}
This state of affairs naturally raises a central question:
does every unital associative zJpd algebra is generated by idempotents?
In particular, in this note we focus on this question in the finite dimensional case,
and we can answer this question affirmatively as follows.

\begin{theorem} \label{thm}
Let $F$ be a field of characteristic not 2 and $A$ be a finite dimensional associative unital $F$-algebra.
Then $A$ is zJpd if and only if it is generated by idempotents.
\end{theorem}

Consequently, the above theorem becomes the Jordan analogue of Bre\v{s}ar's characterization of finite dimensional associative zpd algebras (see \cite[Theorem 3.7]{Bresar}). 
In addition, it clarifies the role played by idempotents in zJpd algebras,
which is not easily discovered in the definition.
Let us end the introduction with a terminological convention: by an {\em algebra}
we will mean an associative unital algebra over some field.

\section{Main Result} 
This section is devoted to proving the main result Theorem \ref{thm}.
Throughout, we assume that $F$ is a field of $\ch (F)\neq 2$.
For an $F$-algebra $A$, we denote $\rad (A)$ and $Z(A)$ the Jacobson radical and center
of $A$ respectively. 

We start our proof by recalling the following equivalent description of zJpd algebras,
which is crucial for the application of zJpd algebras in linear preserver problems.

\begin{proposition}{\rm (\cite[Proposition 1.3]{BOOK})}\label{zJpd-def}
Let $A$ be an $F$-algebra. Then $A$ is zJpd if and only if
for every $F$-vector space $V$ and every bilinear map $f: A\times A\rightarrow V$
with the property $f(a,b)=0$ whenever $ a\circ b=0$,
there exists a linear map $\Phi: A\rightarrow V$ such that $f(x,y)=\Phi(x\circ y)$
for all $x,y\in A$.
\end{proposition}

Taking our target in mind, given a finite dimensional zJpd $F$-algebra $A$,
we want to show that $A$ is generated by idempotents. Notice that
a module theoretic equivalent description of finite dimensional algebras generated
by idempotents was established in \cite{HX}, where one of the main ingredients
depends on the studying of local algebras (see \cite[Lemmas 2.10 and 2.13]{HX}).
This naturally motivates our study of the zJpd property of local algebras.

Recall that a finite dimensional $F$-algebra $A$ is {\em local} if the set of all
non-invertible elements forms an ideal of $A$, which coincides with the Jacobson
radical $\rad (A)$. Equivalently, $A/\rad (A)$ is a division algebra.
Moreover, by lifting idempotents modulo nil ideal, we know that $A$ is local
if and only if it has no nontrivial idempotents.
Accordingly, we need in turn to study the zJpd property of division algebras and then local algebras.

We now recall some properties of finite dimensional division algebras,
see \cite[Section 22]{Draxl}.

Let $A$ be a finite dimensional central simple $F$-algebra. Recall that
an extension field $K$ of $F$ is said to be a {\em splitting field} for $A$ if
$K\otimes_F A\cong M_n(K)$ for some $n\in \mathbb{N}$,
where $M_n(K)$ denotes the $n\times n$ full matrix algebra over $K$.
If we further assume
that $A$ is a division $F$-algebra, then every maximal subfield $K$ of $A$
is a splitting field for $A$. Here we would like to remark that the subfield
$K$ contains the unity of $A$, i.e., $K\supseteq F$.

Let $D$ be a finite dimensional division $F$-algebra.
Then the center $Z(D)$ is a field and hence $D$ is a finite dimensional central division
$Z(D)$-algebra. 
Hence there exists a splitting field $K$ for $D$ such that
$K\otimes_{Z(D)}D\cong M_n(K)$,
where $n=\sqrt{\dim_{Z(D)} D}$ and $K$ is a finite extension of $Z(D)$.
Fix a $Z(D)$-algebra isomorphism
\[
\varphi: K\otimes_{Z(D)}D\longrightarrow M_n(K).
\]
For every $a\in D$, define
\[
\operatorname{RTr}(a)
:=\operatorname{Tr}\bigl(\varphi(1\otimes a)\bigr),
\]
where $\operatorname{Tr}$ denotes the ordinary matrix trace.
By \cite[Section~22, Lemma~2]{Draxl},
the image $\operatorname{RTr}(a)$ in fact belongs to $Z(D)$ and is independent of the
choice of both the splitting field $K$ and the isomorphism
$\varphi$. Consequently,
\[
\operatorname{RTr}:D\longrightarrow Z(D)
\]
is a well-defined $Z(D)$-linear map, called the \emph{reduced trace} of $D$.
Furthermore, it is clear that
$\operatorname{RTr}(xy)=\operatorname{RTr}(yx)$
for all $x,y\in D$. A key property of the reduced trace is as follows.

\begin{proposition}\label{tr1}
{\rm (\cite[Section 22, Lemma 4]{Draxl})}
The reduced trace $\operatorname{RTr}$ is surjective.
\end{proposition}
 
The following equivalent statements of Proposition \ref{tr1} to some extent is well-known
for experts, which means that the finite dimensional division $F$-algebra $D$
is a symmetric $Z(D)$-algebra. In other words, the division algebra
$D$ in some sense is of a geometric
structure over $Z(D)$. We include a proof here for the reader's convenience.
%and because of its independent interest, we  

\begin{proposition}\label{tr2}
Let $D$ be a finite dimensional division $F$-algebra.
Then
\[
B(x,y):=\operatorname{RTr}(xy),\quad \text{for all } x,y\in D,
\]
is a nondegenerate associative symmetric $Z(D)$-bilinear form.
\end{proposition}

\begin{proof}
We only need to prove the non-degeneracy of $B$.
Let $x\in D$ such that $B(x,y)=\operatorname{RTr}(xy)=0$ for all $y\in D$.
If $x\neq 0$, then $x$ is invertible in the division algebra $D$. Hence
\[
0=B(x,x^{-1}y)=\operatorname{RTr}(y)
\]
for all $y\in D$, which is in contradiction with the surjectivity of $\operatorname{RTr}$.
\end{proof}

We also can get the surjectivity of reduced trace $\operatorname{RTr}$ by the
non-degeneracy of $B$. After the above preparation,
we now study the zJpd property of division algebras,
inspirited by \cite[Example 3.19]{BOOK} of real quaternions.

\begin{lemma}\label{di}
If a finite dimensional division $F$-algebra $D$ is zJpd, then $D=F$.
\end{lemma}

\begin{proof}
We first remark that when the base field $F$ is of characteristic not 2,
a commutative algebra is zpd if and only if it is zJpd. If
$D$ is commutative, then $D$ is zpd and hence it is generated by idempotents
by \cite[Theorem 3.7]{Bresar}. Therefore $D=F$, since it has to be generated
by the only nonzero idempotent $1$.

Now assume that $D$ is noncommutative. We claim that
\begin{equation}\label{eq3}
\operatorname{RTr}(u) = 0, \quad \text{if } u \text{ is a Jordan zero divisor in } D.
\end{equation}
In fact, if $u\in D$ is a Jordan zero divisor, then there exists a nonzero element
$v\in D$ such that $u\circ v=0$. We can get that $u = -v^{-1}uv$ and
$$
\operatorname{RTr}(u) = \operatorname{RTr}(-v^{-1}uv) = -\operatorname{RTr}(v^{-1}uv) = -\operatorname{RTr}(u).
$$
As $\ch (F)\neq 2$, the claim follows.

Let us define an $F$-bilinear map  $f: D \times D \to Z(D)$ by 
\[
f(x,y) = \operatorname{RTr}(x)\operatorname{RTr}(y),
\]
for all $x,y\in D$. It follows from the claim (\ref{eq3}) that
$f(u,v)=\operatorname{RTr}(u)\operatorname{RTr}(v)=0$ whenever $u\circ v =0$.
Since $D$ is zJpd, by Proposition \ref{zJpd-def},
there exists an $F$-linear map $\Phi: D\to Z(D)$ such that $f(x,y)=\Phi(x\circ y)$
for all $x,y\in D$. Consequently, 
$$f(x,y)=\Phi(x\circ y)=\frac{1}{2}\Phi((x\circ y)\circ 1)=\frac{1}{2}f(x\circ y,1).$$
%This implies that $f(x,y)=\frac{1}{2}f(x\circ y,1)$ for all $x,y\in D$.
In other words,
\begin{equation}\label{eq4}
\operatorname{RTr}(x)\operatorname{RTr}(y)
=\frac{1}{2}\operatorname{RTr}(x\circ y)\operatorname{RTr}(1), 
\quad \text{for all } x,y \in D.
\end{equation}

On one hand, notice that the $Z(D)$-bilinear (and hence $F$-bilinear) map 
$B(x,y)=\operatorname{RTr}(xy)$
is nondegenerate by Proposition \ref{tr2}. There exists an element $a\in D$ such that
$B(1,a)=\operatorname{RTr}(a)\neq0$. Moreover, the equation (\ref{eq4}) gives
$
\frac{1}{2} \operatorname{RTr}(a\circ a)\operatorname{RTr}(1)
=\operatorname{RTr}(a)\operatorname{RTr}(a) \neq 0,
$
and therefore
\begin{equation}\label{eq5}
\operatorname{RTr}(1) \neq 0.
\end{equation}

On the other hand, since $D$ is noncommutative, 
there exists nontrivial Jordan zero divisors in $D$ by \cite[Corollary 1.8]{BOOK}, i.e.,
there are nonzero elements $u,v\in D$ such that $u \circ v = 0$.
Combining the equations (\ref{eq3}) and (\ref{eq4}), we have
\[
0=\operatorname{RTr}(u)\operatorname{RTr}(u^{-1}) = \frac{1}{2} \operatorname{RTr}(u \circ u^{-1})\operatorname{RTr}(1) = \operatorname{RTr}(1)^2.
\]
Therefore $\operatorname{RTr}(1)=0$, which contradicts the result (\ref{eq5}).
This contradiction shows that $D$ must be commutative and we complete the proof of the lemma.
\end{proof}

As promised, we now turn to local algebras.

\begin{lemma}\label{lo}
If a finite dimensional local $F$-algebra $A$ is zJpd, then $A=F$.
\end{lemma}

\begin{proof}
Since the base field $F$ is of characteristic not 2, then $(A,\circ)$ is
a unital zpd Jordan algebra, then any homomorphic image of $A$
is also zJpd by \cite[Corollary 1.15]{BOOK}, as every homomorphism is a Jordan homomorphism.
We get that $A/\rad(A)$ is a finite dimensional zJpd division algebra
and hence $A/\rad(A)\cong F$ by Lemma \ref{di}.
Furthermore $A=F\oplus \rad(A)$ as the direct sum of $F$-vector space.

We assume in reverse that $\rad(A)\neq0$. 
Since $A$ is finite dimensional, then $\rad(A)$ is a nilpotent ideal and hence
the ideal $\rad(A)^2$ is a proper subset of $\rad(A)$.
Let us define an $F$-bilinear map
$f: A\times A\to \rad(A)/\rad(A)^2$
by
\[
f(\lambda+a,\mu+b)=\lambda b+\rad(A)^2
\]
for all $\lambda,\mu\in F$ and $a,b\in\rad(A)$. Taking arbitrarily a pair of
Jordan zero divisors with $u\circ v=0$, we write $u=\lambda+a$ and $v=\mu+b$
with $\lambda,\mu\in F$ and $a,b\in\rad(A)$. It is clear that 
$2\lambda\mu+2\lambda b+2\mu a+a\circ b=0$ and furthermore $\lambda\mu=0$,
since $\ch(F)\neq 2$.
If $\lambda=0$, obviously $f(\lambda+a,\mu+b)=0$.
If $\mu=0$, then $\lambda b=-\frac12(a\circ b)\in \rad(A)^2$, and we can also get
$f(\lambda+a,\mu+b)=0$.
Hence the bilinear map $f$ satisfies
$f(u,v)=0$ whenever $u\circ v=0$.
Since $A$ is zJpd, it follows from Proposition \ref{zJpd-def} that
there exists an $F$-linear map $\Phi: A\to \rad(A)/\rad(A)^2$
such that $f(x,y)=\Phi(x\circ y)$ for all $x,y\in A$.
Choose $a\in \rad(A)\setminus \rad(A)^2$ and we have
\[
f(1,a)=\Phi(1\circ a)=\Phi(a\circ1)=f(a,1).
\]
However,
\[
f(1,a)
=
a+\rad(A)^2
\neq0,
\]
while $f(a,1)=0$.
This contradiction shows that $\rad(A)=0$. Hence $A=F$.
\end{proof}

The remaining tool we need is the theory of lifting idempotents modulo ideals, see \cite{T,N}.
The classical version of this theory concerns lifting idempotents modulo nilpotent ideals.
However, in the case of finite dimensional algebras, lifting idempotents
can be made modulo any ideals due to the work of Nicholson \cite{N}.
We refer the reader to \cite[Lemma 2.25]{BOOK} for a self-contained proof
of this fact.
 
\begin{lemma}\label{l1}
Let $A$ be a finite dimensional unital algebra and $I$ be an ideal of $A$. 
Then every idempotent in $A/I$ is of the form $e + I$ with $e\in A$ an idempotent.
\end{lemma}

We are now in a position to prove the main theorem.

\begin{proof}[Proof of Theorem \ref{thm}]
We only need to prove the necessity.
Let $A$ be a finite dimensional zJpd $F$-algebra. From now on, We denote $I$ the ideal of $A$
generated by all commutators of the form $[e,x]$, where $e$ is an idempotent and $x$
is an arbitrary element of $A$. Let $E$ be the subalgebra of $A$ generated by all
idempotents. We want to show $A=E$.
 
The proof is divided into three steps.

\medskip
\noindent
{\bf Step 1.} 
If we further assume that every idempotent in $A$ is central, then $A\cong F\oplus \cdots 
\oplus F$. 

Let $1= e_1 + \cdots + e_r$ be the decomposition of unity into the sum of primitive idempotents. Since every idempotent is central, this decomposition is also
corresponding to the block decomposition of $A$. In other words, 
$A\cong e_1 Ae_1\oplus \cdots\oplus e_r Ae_r$ and each block
$e_i Ae_i\cong \End_A(Ae_i)$ is local. Notice that $e_i Ae_i$ is
a zJpd $F$-algebra for $1\leq i\leq r$ by \cite[Theorem 1.16]{BOOK}.
Hence the result follows from Lemma \ref{lo}.

\medskip
\noindent
{\bf Step 2.}
We show that $A=E+\rad(A)$.

Similar to the first paragraph of the proof of Lemma \ref{lo}, we get that
$A/\rad(A)$ is a finite dimensional zJpd semisimple algebra.
The well-known Wedderburn-Artin theorem implies that
\[
A/\rad(A)\cong \bigoplus_{i=1}^{r}M_{n_i}(D_i),
\]
where $n_i\in \mathbb{N}$ and each $D_i$ is a division $F$-algebra.
Notice that each algebra $M_{n_i}(D_i)$ is zJpd by \cite[Theorem 1.16]{BOOK} again.
Hence it follows from Lemma \ref{di} that
\[
A/\rad(A)\cong
M_{n_1}(D_1)\oplus\cdots\oplus
M_{n_k}(D_k)\oplus F\oplus\cdots\oplus F,
\]
where $n_i\ge 2$ for each $1\leq i\leq k$.
It is well-known that the algebra $M_{n}(D)$ is generated by idempotents if $n\ge 2$,
see \cite[Corollary 2.4]{BOOK} for example.
Therefore, every element $x+\rad(A)\in A/\rad(A)$ is a linear combination of products of
idempotents in $A/\rad(A)$, Using Lemma \ref{l1} we can find an element $x_0\in E$
such that $x+\rad(A)=x_0+\rad(A)$,
which in turn means $A=E+\rad(A)$.

\medskip
\noindent
{\bf Step 3.}
We finally show that $\rad(A)\subseteq E$ and hence $A=E$ as desired.

Recall that $I$ is the ideal generated by all commutators of idempotents with
arbitrary elements in $A$.
By Lemma \ref{l1} again,
every idempotent of $A/I$ is of the form $e+I$ with $e\in A$ an idempotent.
Since $[e, x]\in I$ for all $x\in A$, it follows that every idempotent in
$A/I$ is central. On the other hand, the quotient algebra $A/I$ is also
zJpd by \cite[Corollary 1.15]{BOOK}. We have
\[
A/I\cong F\oplus \cdots \oplus F
\]
by Step 1. In particular $A/I$ is semisimple.
It follows that $\rad(A)\subseteq I$. Notice that $I\subseteq E$, see
\cite[Lemma 2.26]{BOOK} for example. We have $\rad(A)\subseteq E$
and complete the proof of the main result.
\end{proof}

Combining Theorem \ref{thm} and Bre\v{s}ar's Theorem 3.7 in \cite{Bresar}
(see also \cite[Theorem 2.27]{BOOK}), we obtain that for finite dimensional
$F$-algebras, the notion of a zpd algebra coincides with that of a zJpd algebra,
provided that the base field $F$ is of characteristic not 2, which is a quite common
restriction in the study of Jordan and Lie algebras.
As far as we know, researchers have not found an example of zpd or zJpd algebras
which is not generated by idempotents. We tend to have a such example,
but we still have 

\begin{conjecture} \label{open}
Let $A$ be an $F$-algebra (maybe infinite dimensional).
Then $A$ is zJpd if and only if it is zpd.
\end{conjecture}

Let us end this note with an open problem of zero Lie product determined algebras.

\begin{remark}
For an associative $F$-algebra $A$, denoting the Lie
product on $A$ as usual by $[x, y]=xy-yx$, then $(A,[\cdot,\cdot])$
becomes a Lie algebra. Recall that $A$ is said to be \emph{zero Lie product determined}
(\emph{zLpd} for short) if the Lie algebra $(A,[\cdot,\cdot])$ is zero product determined.

Unlike the zpd and zJpd algebras, 
the zLpd algebras exhibit a substantially different behavior, see \cite[Section 3.2]{BOOK}.
Even every finite dimensional complex simple Lie algebra is zpd \cite{WangYC},
we have not found a representation theoretic description of this fact, see \cite{BGLLZ}.
However, in our opinion, for a zLpd algebra $A$,
there should exist and hence we ask for a structural property about the generators of $(A,[\cdot,\cdot])$
modulo its Lie ideal $Z(A)$,
%analogous to Theorem \ref{thm}, 
which analogously can be described by the generators (as an associative algebra) 
of the universal envelop algebra of $(A,[\cdot,\cdot])$.

\end{remark}

\smallskip
\noindent{\bf Data Availability Statement.} All data generated or analyzed during this study are included in this published article. No additional data are available.

\smallskip
\noindent{\bf Conflict of interest.} On behalf of all authors, the corresponding author states that there is no conflict of
interest.

%\noindent{\bf Acknowledgements}.
%The authors would like to thank the referees for their valuable comments and suggestions
%which significantly helped us improve the final presentation of this paper.

%{\bf Acknowledgements.}

\end{document}